\documentclass[11pt]{amsart}
\usepackage{amsmath, amsthm, amssymb,amscd}
\usepackage{MnSymbol}
\usepackage{float}
\usepackage{graphicx}
\usepackage{xypic}
\usepackage[arrow, matrix, curve]{xy}
\usepackage{color}
\usepackage{enumitem}  
\usepackage{tikz-cd} 
\usepackage{url}

\usepackage{blindtext}

\usepackage[english]{babel}
\newcommand{\seq}{\subseteq}
\newcommand{\C}{\mathbb{C}}

\newtheorem{thm}{Theorem}[section]
\newtheorem*{thm-nl}{Theorem}
\newtheorem*{prop-nl}{Proposition}

\def\PP{{\mathbf P}}

\def\Pic0{{\rm Pic}^0(X)}

\newtheorem*{cor-nl}{Corollary}
\newtheorem{conjecture}[thm]{Conjecture}
\newtheorem*{conjecture-nl}{Conjecture}

\newtheorem*{quest-nl}{Question}
\newtheorem*{quests-nl}{Questions}

\newtheorem{prop}[thm]{Proposition}

\theoremstyle{remark}

\title{{Linear syzygies of projective space}}
\date{\today}

\author[M. Kemeny]{Michael Kemeny}

\address{University of Wisconsin-Madison, Department of Mathematics, 480 Lincoln Dr
\hfill \newline\texttt{}
 \indent WI 53706, USA} 
 \email{{\tt michael.kemeny@gmail.com}}

\begin{document}
\begin{abstract}
In this paper, we study the Ein--Lazarsfeld Conjecture of syzygies of Veronese varieties $\PP^n$ embedded by $\mathcal{O}_{\PP^n}(d)$. We show a vanishing statement which agrees with the conjecture up to highest and second-highest order for linear syzygies $q=1$.
\end{abstract}

\maketitle
\setcounter{section}{-1}
\section{Introduction}
In this paper, we study syzygies for $\PP^n$ embedded by $\mathcal{O}_{\PP^n}(d)$. The case $n=1$ is known, as then the resolution is an Eagon--Northcott complex \cite[Ch.\ 6]{eisenbud-syzygies}, so we take $n \geq 2$. The Ein--Lazarsfeld Conjecture of syzygies of a Veronese variety, \cite[Theorem B]{ein-lazarsfeld-asymptotic} says 
\begin{conjecture}[Ein--Lazarsfeld] \label{EL-conj}
Fix an index $1 \leq q \leq n$ and $d \geq n+1$, then
$$\mathrm{K}_{p,q}(\PP^n,\mathcal{O}_{\PP^n}(d)) \neq 0$$
if and only if 
$$\binom{d+q}{q}-\binom{d-1}{q}-q\leq p \leq \binom{d+n}{n}-\binom{d+n-q}{n-q}+\binom{n}{n-q}-q-1.$$
\end{conjecture}
The nonvanishing direction of the statement was shown in \cite{ein-lazarsfeld-asymptotic}, so we focus on the vanishing, which was previously considered by Ottaviani-Paoletti \cite{ott-pao}. Syzygies of Veronese varieties has been much studied in recent years, see for example \cite{park}, \cite{erman-et-al}, \cite{raicu}.\\

We first remark that existing results of Aprodu \cite{aprodu-higher}, Green \cite{green-koszul} and Birkenhake \cite{birkenhake} give Conjecture \ref{EL-conj} in the case $n=2$.
\begin{prop}[Aprodu, Green, Birkenhake]  \label{plane-thm}
Conjecture \ref{EL-conj} holds in the case $n=2$.
\end{prop}
We then show how the vanishings from this case generalize to statements about the Ein--Lazarsfeld Conjecture for $q=1$ and $q=n$, for all $n \geq 2$.
In the case $q=1$ of \emph{linear} syzygies, which have been studied in, for instance \cite{lin-syz} and \cite{ein-lazarsfeld}, the conjecture predicts the following vanishing:
\begin{conjecture} \label{EL-conj-linear}
Fix $n \geq 2$. For $d \geq n+1$, then
$$\mathrm{K}_{p,1}(\PP^n,\mathcal{O}_{\PP^n}(d)) = 0$$
if 
$$p \geq \binom{d+n-1}{n}+n-1.$$
\end{conjecture}
To obtain this from conjecture \ref{EL-conj}, we used Pascal's Triangle
$$\binom{d+n}{n}-\binom{d+n-1}{n-1}=\binom{d+n-1}{n}.$$

In this paper, we prove:
\begin{thm}  \label{main-thm}
Fix $n \geq 3$. For $d \geq 1$, then
$$\mathrm{K}_{p,1}(\PP^n,\mathcal{O}_{\PP^n}(d)) = 0$$
if 
$$p \geq \binom{d+n-1}{n}+\binom{d+n-2}{n-2}.$$
\end{thm}
Notice that the Theorem holds for any $d \geq 1$, whereas conjecture \ref{EL-conj-linear} requires $d \geq n+1$. Indeed, the Theorem for $d=1$ is very well known \cite[Proposition 1.5]{aprodu-nagel}, whereas for $n=2$, the result of conjecture \ref{EL-conj-linear} was already known for $d\geq 2$, \cite[Lemma 5.1]{aprodu-higher}. The nonvanishing arguments, on the other hand, in conjecture \ref{EL-conj}, require $d\geq n+1$. Also note that theorem  \ref{main-thm} requires $n\geq 3$ to even be stated.\\

 The bound in Theorem \ref{main-thm} is in general weaker than in the conjecture, but they only differ for terms of order less or equal to the third highest most order in $d$. For instance, for $n=3$, Theorem \ref{main-thm} states $\mathrm{K}_{p,1}(\PP^n,\mathcal{O}_{\PP^n}(d)) = 0$ for $p \geq \binom{d+2}{3}+d+n-1$ whereas the conjecture is $\mathrm{K}_{p,1}(\PP^n,\mathcal{O}_{\PP^n}(d)) = 0$ for $p \geq \binom{d+2}{3}+2$ (for $d\geq 4$).\\

 The proof of Theorem \ref{main-thm} uses the projection approach of Ehbauer \cite{ehbauer} and Aprodu \cite{aprodu-higher} to show Theorem \ref{main-thm} by induction on $d$, reduces to the trivial case $d=1$. The projection approach to syzygies has been used in \cite{projecting} and \cite{choi-kang-kwak}. As we explain in our proof, our result is optimal for our method of projection.\\

\textbf{Acknowledgements} I thank G.\ Farkas for a discussion on this topic. The author was supported by NSF grant DMS-2100782 and a Vilas Life Cycle award.
\smallskip

\section{Construction of Projection of Syzygies}
We begin with a construction of the projection and evaluation maps on syzygies, as in \cite[\S 2.2]{aprodu-nagel}, however we give it in more generality. We first construct and define the \emph{evaluation} map rather than the projection map, which may be constructed without an added assumption used in our construction of the projection map.
Let $X$ be a smooth, complex variety, let $N$ be a base-point free line bundle on $X$, and let $D \seq X$ be an effective divisor and assume further that the evaluation map
$$\mathrm{H}^0(X,N) \to \textrm{H}^0(N_{|_D})$$
is surjective. Let $L:=N(-D)$. Assume $L$ is also base point free on $X$. Consider the \emph{kernel bundles} $M_L$ and $M_N$ as in \cite[\S 2.1]{aprodu-nagel}, which fit into a commutative diagram
{\small{$$\begin{tikzcd}
& 0 \arrow[d] & 0 \arrow[d]  & 0 \arrow[d] \\
0 \arrow[r] & M_{L} \arrow[r] \arrow[d] & \mathrm{H}^0(X,L) \otimes \mathcal{O}_X \arrow[r] \arrow[d] &  L=N(-D) \arrow[r] \arrow[d] &0\\
0 \arrow[r] & M_N \arrow[r] \arrow[d] & \mathrm{H}^0(X,N) \otimes \mathcal{O}_X  \arrow[r] \arrow[d]  &  N \arrow[r] \arrow[d]& 0\\
0 \arrow[r] & \Gamma \arrow[r] \arrow[d] & \textrm{H}^0(N_{|_D})  \otimes \mathcal{O}_X  \arrow[r] \arrow[d]  &  N_{|_D}  \arrow[r] \arrow[d] & 0\\
& 0 & 0 & 0 &
\end{tikzcd},$$}}where $\Gamma$ is a vector bundle (since $D$ is a divisor) of rank $s:=\textrm{h}^0(N_{|_D})=\dim \textrm{H}^0(N_{|_D}).$ In particular we have the exact sequence of vector bundles
$$0 \to  M_{L} \to M_N \to \Gamma \to 0,$$
to which we can apply \cite[II. Exercise 5.16]{hartshorne}. For any integer $p>s$, we have an exact sequence of sheaves
$$0 \to F \to \bigwedge^p M_N \to \bigwedge^{p-s} M_L(-D)\to 0,$$
where $F$ is a coherent sheaf, using $\bigwedge^s \Gamma \simeq \mathcal{O}_X(-D).$ We have natural inclusions $M_L(-D) \seq M_L \seq M_N$, and so the surjective morphism $\bigwedge^p M_N \to \bigwedge^{p-s} M_L(-D)$ induces a morphism of sheaves $$ t_p\; : \; \bigwedge^p M_N \to \bigwedge^{p-s} M_N.$$ Up to scalar, the morphism $t_p$ can be described as such. We have a natural inclusion $\bigwedge^p M_N \seq \bigwedge^p \mathrm{H}^0(N) \otimes \mathcal{O}_X$, and there is a map
\begin{align*}
\alpha_p \; : \; \bigwedge^p \mathrm{H}^0(N)\otimes \mathcal{O}_X &\to \bigwedge^{p-s} \mathrm{H}^0(N) \otimes \mathcal{O}_X \\
v_1 \wedge \ldots \wedge v_p &\mapsto \sum_{\{i_1, \ldots, i_s\} \seq\{1, \ldots,p\}} (-1)^{i_1+\ldots+i_s}v_1 \wedge \ldots \wedge \hat{v}_{i_1} \wedge \ldots \wedge \hat{v}_{i_s}\wedge \ldots \wedge v_p \otimes \gamma(v_{i_1}\wedge \ldots \wedge_{i_s}),
\end{align*}
for $\gamma: \bigwedge^{s} \mathrm{H}^0(N)\otimes \mathcal{O}_X\to  \bigwedge^{s}\textrm{H}^0(N_{|_D}) \otimes \mathcal{O}_X\simeq \mathcal{O}_X$ (which is only uniquely defined up to scalar) and $\alpha_p (\bigwedge^p M_N) \seq \bigwedge^{p-s} M_N$. We have $t=\alpha_{|_{\bigwedge^p M_N}}$. After choosing the sign of the scalar defining $\gamma$ correctly, we have a commutative diagram 
{\small{$$\begin{tikzcd}
 \bigwedge^{p+1} \mathrm{H}^0(N)\; \; \; \arrow[d, "\beta_p"] \arrow[r, "\mathrm{H}^0(\alpha_{p+1})"] & \; \; \; \;  \bigwedge^{p+1-s} \mathrm{H}^0(N) \arrow[d, "\beta_{p-s}"]  & \\
  \mathrm{H}^0(\bigwedge^{p} M_{N}(N))\; \; \;\arrow[r, "\mathrm{H}^0(t_{p}\otimes \mathrm{id}_N)"]  & \;  \;  \; \; \mathrm{H}^0(\bigwedge^{p-s} M_{N}(N)) 
\end{tikzcd},$$}}
where the vertical maps are induced from taking global sections of the exact sequence
$$0 \to \bigwedge^i M_N \to  \bigwedge^i \mathrm{H}^0(N)\otimes \mathcal{O}_X \xrightarrow{b_i} \bigwedge^{i-1} M_N\otimes N \to 0,$$
for $i \in \{p+1, p+1-s \}$. Similarly to \cite[Remark 1.3]{aprodu-nagel}, the morphism of sheaves $b_i$ can be described as
$$v_1 \wedge \ldots \wedge v_i \mapsto \sum_j (-1)^j v_1 \wedge \ldots \wedge \hat{v}_j \wedge \ldots \wedge v_i \otimes \mathrm{ev}(v_j),$$
where $\mathrm{ev}$ is the evaluation morphism 
$$ \mathrm{ev} \; : \; \mathrm{H}^0(N) \otimes \mathcal{O}_X \to N.$$ From the commutative diagram, we get a morphism $$\mathrm{ev}_{p,D} \; : \; \mathrm{K}_{p,1}(X,N) \simeq \mathrm{Coker}(\beta_p) \to  \mathrm{K}_{p-s,1}(X,N) \simeq \mathrm{Coker}(\beta_p),$$
called the \emph{evaluation map}. When $s=1$, it coincides with the map defined in \cite[\S 2.2]{aprodu-nagel}. Up to scalar, it is the composition of $s$ evaluation maps
$$\mathrm{K}_{p-j,1}(X,N) \to \mathrm{K}_{p-j-1,1}(X,N), \; j \in \{1, \ldots, s\},$$
as in  \cite[\S 2.2]{aprodu-nagel}, with the morphism
$$\mathrm{H}^0(X,N) \to \mathrm{H}^0(X,N_{|_D})\simeq \C^s \to \C,$$
induced by the projection to the $s$ coordinates of $\C^s$.\\

In the case where, in addition $\mathrm{H}^1(\mathcal{O}_X)=0$, we can construct more. Then $\mathrm{K}_{p,1}(X,N) \simeq \mathrm{H}^1(\bigwedge^{p+1}M_N)$ and $\mathrm{H}^1(t_{p+1})$ induces, up to sign, the evaluation map
$$ \mathrm{K}_{p,1}(X,N)  \to  \mathrm{K}_{p-s,1}(X,N) .$$
Indeed we have a commutative diagram
{\small{$$\begin{tikzcd}
0 \arrow[r] &\bigwedge^{p+1} M_N \arrow[r] \arrow[d, "t_{p+1}"] &\bigwedge^{p+1} \mathrm{H}^0(N)\otimes \mathcal{O}_X \arrow[r, "b_{p+1}"] \arrow[d, "\alpha_{p+1}"]  &\bigwedge^{i-1} M_N\otimes N \arrow[r] \arrow[d, "t_p\otimes\mathrm{id}"] &0 \\
0 \arrow[r]  &\bigwedge^{p+1-s} M_N \arrow[r] &\bigwedge^{p+1-s} \mathrm{H}^0(N)\otimes \mathcal{O}_X \arrow[r, "b_{p+1-s}"]  &\bigwedge^{i-1} M_N\otimes N \arrow[r] &0 
\end{tikzcd}$$}}with exact rows, and then the claim follows from standard results about the functoriality of cohomology. As we saw above, the image of $t_{p+1}$ lands in $\bigwedge^{p+1-s}M_L(-D)$, so the evaluation map factors through a map
$$\textrm{pr}_{p,D} \; : \;  \\  \mathrm{K}_{p,1}(X,N)  \to  \mathrm{K}_{p-s,1}(X,L),$$
called the \emph{projection map} (this can be seen to coincide with the projection map of \cite{aprodu-nagel}, when $s=1$ and $M$ is the coordinate ring, using Ehbauer's theorem \cite{aprodu-higher}). Further, the image of $\textrm{pr}_{p,D}$ is contained in the image of $ \mathrm{H}^1(\bigwedge^{p+1-s}M_L(-D))) \to  \mathrm{H}^1(\bigwedge^{p+1-s}M_L) \simeq \mathrm{K}_{p-s,1}(X,L)$.

\section{Linear syzygies}
We first prove Proposition \ref{plane-thm} for the projective plane, i.e.\ $n=2$. For $q=2$ and $n=2$, the vanishing direction of the  Ein--Lazarsfeld Conjecture is equivalent to two statements, both known. The first is
$$\mathrm{K}_{p,2}(\PP^2,\mathcal{O}_{\PP^2}(d)) = 0$$
for $p > \binom{d+2}{2}-3$.\\

 By the Duality Theorem, \cite{green-koszul}, this statement is equivalent to 
$$\mathrm{K}_{\binom{d+2}{2}-3-p,1}(\PP^2,\mathcal{O}_{\PP^2}(-3);\mathcal{O}_{\PP^2}(d)) = 0$$
for $p > \binom{d+2}{2}-3$ which is trivial. The second statement is 
$$\mathrm{K}_{p,2}(\PP^2,\mathcal{O}_{\PP^2}(d)) = 0$$
for $p < 3d-2$, which is a result of Green and Birkenhake \cite{birkenhake}, with a statement generalized to a conjecture on condition $N_p$ for all $n$ by Ottaviani--Paoletti \cite{ottaviani-paoletti}.\\

So then, to prove the remainder of Conjecture \ref{EL-conj} in the case $n=2$, we must proving a vanishing of linear syzygies, namely we must prove the statement
$$\mathrm{K}_{p,1}(\PP^2,\mathcal{O}_{\PP^2}(d))=0,$$
for $d\geq 3$ and $p \geq \frac{1}{2}(d+2)(d+1)-d.$ 
\begin{proof}
As explained in the introduction, we need to show
$$\mathrm{K}_{p,1}(\PP^2,\mathcal{O}_{\PP^2}(d))=0,$$
for $d\geq 3$ and $p \geq \frac{1}{2}(d+2)(d+1)-d.$ This is \cite[Lemma 5.1]{aprodu-higher}.
\end{proof}

We now prove Theorem \ref{main-thm}.
\begin{proof}[Proof of Theorem \ref{main-thm}]
Fix an integer $n \geq 3$ and fix $d \geq 1$. We will apply the construction in the last section for $N=\mathcal{O}_{\PP^n}(d)$ and $D\in |\mathcal{O}_{\PP^n}(1)|$. Set $$s=h^0(\mathcal{O}_{\PP^n}(d))-h^0(\mathcal{O}_{\PP^n}(d-1))=h^0(\mathcal{O}_{\PP^{n-1}}(d))$$
by Pascal's triangle and consider $s$ general points $p_1, \ldots, p_s \in D$. Let $H \seq \PP(\mathrm{H}^0(\mathcal{O}_{\PP^n}(d)))$ be a hyperplane containing $p_1, \ldots, p_s$. 
As seen in the last section, we can view the evaluation map $$\mathrm{ev}_{p,D}: \mathrm{K}_{p,1}(\PP^n, \mathcal{O}_{\PP^n}(d)) \to \mathrm{K}_{p-s,1}(\PP^n, \mathcal{O}_{\PP^n}(d)) $$
as a composite of $s$ evaluation maps. By Aprodu \cite[\S 2, \S 3]{aprodu-higher}, $\mathrm{ev}_{p,D}$ is nonzero if we assume $ \mathrm{K}_{p,1}(\PP^n, \mathcal{O}_{\PP^n}(d))\neq 0$.  Since $\mathrm{H}^1(\mathcal{O}_{\PP^n})=0$, the image of $\mathrm{ev}_{p,D}$ is contained in the image of
$$\mathrm{H}^1(\bigwedge^{p+1-s}M_{\mathcal{O}_{\PP^n}(d-1)}(-1)) \to  \mathrm{K}_{p-s,1}(\PP^n, \mathcal{O}_{\PP^n}(d)) .$$
So $\mathrm{K}_{p,1}(\PP^n, \mathcal{O}_{\PP^n}(d)) \neq 0 \implies \mathrm{H}^1(\bigwedge^{p+1-s}M_{\mathcal{O}_{\PP^n}(d-1)}(-1)) \neq 0.$ But $\mathrm{H}^1(\mathcal{O}_{\PP^n}(-1)=0$, so $$ \mathrm{H}^1(\bigwedge^{p+1-s}M_{\mathcal{O}_{\PP^n}(d-1)}(-1)) \simeq \mathrm{K}_{p-s,1}(\PP^n; \mathcal{O}(-1), \mathcal{O}(d-1)),$$
by \cite[Remark 2.6]{aprodu-nagel} \footnote{Note that there is a typo in this book, the right hand side should state $\bigwedge^{p+1}W \otimes \mathrm{H}^1(X, F \otimes L^{q-1})$.}. But by the Vanishing Theorem \cite{green-koszul}, $\mathrm{K}_{p-s,1}(\PP^n; \mathcal{O}(-1), \mathcal{O}(d-1))=0$ for $p-s \geq \binom{d-2+n}{n}$. Note that $\mathrm{K}_{p-s,1}(\PP^n; \mathcal{O}(-1), \mathcal{O}(d-1))=\mathrm{K}_{p-s,0}(\PP^n; \mathcal{O}(d-2), \mathcal{O}(d-1))$, so this result is optimal by \cite[Remark 6.5]{ein-lazarsfeld-asymptotic}. But this gives
\begin{align*}
p &\geq \binom{d-2+n}{n}+s \\
&=\binom{d-2+n}{n}+\binom{d+n}{n}-\binom{d+n-1}{n}\\
&=\binom{d-2+n}{n}+\binom{d+n-1}{n-1}\\
&=\binom{d-1+n}{n}+\binom{d+n-1}{n-1}-\binom{d+n-2}{n-1}\\
&=\binom{d-1+n}{n}+\binom{d+n-2}{n-2}
\end{align*}
by Pascal's triangle, as required.
\end{proof}

We end this paper by considering Veronese syzygies in the other extreme, with $q=n$. In this case, one side of the vanishing statement is $$\mathrm{K}_{p,n}(\PP^n,\mathcal{O}_{\PP^n}(d)) = 0,$$ if 
$$p \geq \binom{d+n}{n}-n.$$ 
 By the Duality Theorem, \cite{green-koszul}, this statement is equivalent to 
$$\mathrm{K}_{\binom{d+n}{n}-n-1-p,1}(\PP^2,\mathcal{O}_{\PP^2}(-n-1);\mathcal{O}_{\PP^n}(d)) = 0$$
for $p \geq \binom{d+n}{n}-n$ which is trivial. So the vanishing in this case amounts to the following theorem of Ein--Lazrsfeld \cite[Remark 6.5]{ein-lazarsfeld-asymptotic} (in more generality), which we given a different proof of.
\begin{thm}[Ein--Lazarsfeld] \label{q=n-thm}
Fix $n \geq 2$. For $d \geq n+1$, then
$$\mathrm{K}_{p,n}(\PP^n,\mathcal{O}_{\PP^n}(d)) = 0$$
if 
$$p \leq \binom{d+n}{n}-\binom{d-1}{n}-n-1.$$
\end{thm}
\begin{proof}
Our proof of this theorem generalizes Ottaviani--Paoletti's proof of the statement in the special case $n=2$, \cite[Proposition 3.2]{ott-pao}. Note that this theorem requires $d\geq {n+1}$ to be stated so that $\binom{d-1}{n}$ has it's usual meaning.
The goal of this section is to prove Theorem \ref{q=n-thm}. 
Consider a smooth hypersurface $X \in |\mathcal{O}_{\PP^n}(d)|$. By the Lefschetz Theorem \cite{green-koszul}, it is equivalent to prove 
$$\mathrm{K}_{p,n}(X,\mathcal{O}_{X}(d)) = 0$$
for
$$p \leq \binom{d+n}{n}-\binom{d-1}{n}-n-1,$$
where $\mathcal{O}_{X}(d)$ denotes $\mathcal{O}_{\PP^n}(d)|_X$. Notice, by the adjunction formula, $$\omega_X \simeq \mathcal{O}_{X}(d-n-1).$$ By the duality theorem, whose assumptions clearly hold, we need
$$\mathrm{K}_{s-n-p,0}(X,\mathcal{O}_{X}(d-n-1); \mathcal{O}_{X}(d)) = 0$$
for $$s=h^0(\mathcal{O}_{X}(d))-1= \binom{d+n}{n}-1.$$
The claim now follows from the Vanishing Theorem  \cite{green-koszul}.
\end{proof}

\end{document}